\documentclass[leqno,12pt]{amsart} 
\usepackage[mathscr]{eucal}
\usepackage{mathrsfs}
\usepackage{amscd}
\usepackage{amsfonts}
\usepackage{amsmath, amsthm, amssymb}
\usepackage{latexsym}
\usepackage{bm}
\usepackage[dvips]{graphics}
\numberwithin{equation}{section}

\theoremstyle{plain} %
\newtheorem{theorem}{\indent Theorem}[section] %
\newtheorem{lemma}[theorem]{\indent Lemma}
\newtheorem{corollary}[theorem]{\indent Corollary}
\newtheorem{proposition}[theorem]{\indent Proposition}

\theoremstyle{definition} %
\newtheorem{definition}[theorem]{\indent\rm Definition}
\newtheorem{remark}[theorem]{\indent\rm Remark}
\newtheorem{example}[theorem]{\indent\rm Example}

\newcommand{\RII}{\ensuremath{\rm I \! I}}
\newcommand{\trace}{\mathop{\mathrm{trace}}\nolimits}

\def\tr#1{\mathord{\mathopen{{\vphantom{#1}}^t}#1}} 
\def\C{{\mathbb{C}}}
\def\R{{\mathbb{R}}}
\def\H{{\mathbb{H}}}
\def\Z{{\mathbb{Z}}}
\def\N{{\mathbb{N}}}
\def\D{{\mathbb{D}}}
\def\Pi{{\mathbb{P}}}
\def\Si{{\mathbb{S}}}

\title[The hyperbolic Gauss map]{Ramification estimates for the hyperbolic Gauss map}
\author[Yu Kawakami]{Yu Kawakami}
\dedicatory{Dedicated to Professor Junjiro Noguchi on his sixtieth birthday}

\thanks{Partly supported by Global COE program (Kyushu university) ``Education and Research Hub for Mathematics-for-Industry'', 
Osaka City University Advanced Mathematical Institute (OCAMI) and Graduate school of 
Mathematics, Nagoya university.}
\address{
Graduate School of Mathematics, \endgraf
Kyushu university, \endgraf
6-10-1,Hakozaki,Higashiku,
Fukuoka-city,
812-8581, Japan}
\email{kawakami@math.kyushu-u.ac.jp}

\begin{document}

\maketitle

\begin{abstract}
We give the best possible upper bound on the number of exceptional values and the totally ramified value number
of the hyperbolic Gauss map for pseudo-algebraic constant mean curvature one surfaces in the hyperbolic three-space 
and some partial results on the Osserman problem for algebraic case. 
Moreover, we study the value distribution of the hyperbolic Gauss map for complete constant mean curvature one 
faces in de Sitter three-space. 
\end{abstract}

\section*{Introduction}
There exists a representation formula for constant mean curvature one (CMC-$1$, for short) surfaces in 
the hyperbolic $3$-space ${\H}^{3}$ as an analogy of the 
Enneper-Weierstrass formula in minimal surface theory (\cite{Br}, \cite{UY1}). As a result, 
the hyperbolic Gauss map $G$ of CMC-$1$ surfaces in ${\H}^{3}$ have some properties 
similar to the Gauss map $g$ of minimal surfaces in Euclidean $3$-space ${\R}^{3}$. 
In fact, by means of the representation formula, the hyperbolic Gauss map of these surfaces can 
be defined as a holomorphic map to the Riemann sphere $\widehat{\C}:=\C\cup \{\infty\}$. 
These results enable us to use the complex function theory for studying CMC-$1$ surfaces in ${\H}^{3}$. 
In particular, by applying the Fujimoto theorem \cite{Fuj1}, Z.~Yu \cite{Yu} showed that the hyperbolic Gauss map 
of a non-flat complete CMC-$1$ surface can omit at most four values. Moreover, for non-flat complete CMC-$1$ surfaces with 
finite total curvature, Collin, Hauswirth and Rosenberg \cite{CHR2} proved that $G$ can omit at most three values 
by using the result of \cite{UY2}. 
The result corresponds to the Osserman result \cite{O1} for algebraic minimal surfaces (By an algebraic minimal surface, 
we mean a complete minimal surface with finite total curvature). 

On the other hand, the author, Kobayashi and Miyaoka \cite{KKM} refined the Osserman argument and gave an effective estimate for the number of 
exceptional values $D_{g}$ and the totally ramified value number ${\nu}_{g}$ of the Gauss map for pseudo-algebraic and algebraic 
minimal surfaces in ${\R}^{3}$ recently. It also provided new proofs of 
the Fujimoto and the Osserman theorems for these classes and revealed the geometric meaning behind it. 
We \cite{Ka2} also gave such an estimate for them in Euclidean $4$-space ${\R}^{4}$. Jin and Ru \cite{JR} gave the estimate for 
the totally ramified hyperplane number of the generalized Gauss map of algebraic minimal surfaces in Euclidean $n$-space ${\R}^{n}$. 

In this paper, we give the upper bound on the number of exceptional values $D_{G}$ 
and the totally ramified value number ${\nu}_{G}$ of the hyperbolic Gauss map. 
In Section $1$, we recall some fundamental properties and notations about CMC-$1$ surfaces in ${\H}^{3}$.  
In particular, using the notion of ``duality'', 
we define algebraic and pseudo-algebraic CMC-$1$ surfaces and give examples which play important role 
in the following sections (see Example \ref{Vosscousin}). 
In Section $2$, we give the upper bound on the totally ramified value number of 
the hyperbolic Gauss map for pseudo-algebraic and algebraic CMC-$1$ surfaces (Theorem \ref{Main1}). 
This estimate is effective in the sense that the upper bound which we obtain is described in terms of 
geometric invariants and sharp for some topological cases. 
Note that this estimate corresponds to the defect relation in the Nevanlinna theory (see \cite{Ko} and \cite{NO}). 
Moreover, as a corollary of Theorem \ref{Main1}, we give some partial results on the Osserman problem for algebraic CMC-$1$ 
surfaces, that is, give the best possible upper bound on $D_{G}$ for these surfaces. 
Furthermore, applying the classification of algebraic CMC-$1$ surfaces with dual total absolute curvature equal to $8\pi$ \cite{RUY1}, 
we show that the hyperbolic Gauss map of algebraic CMC-$1$ surfaces 
has a value distribution theoretical property which has no analogue in the theory of the Gauss map of minimal surfaces 
in ${\R}^{3}$ (Proposition \ref{main2}). 
In Section $3$, we study the value distribution of the hyperbolic Gauss map of CMC-$1$ faces in de Sitter $3$-space ${\Si}^{3}_{1}$. 
Fujimori \cite{Fu} defined spacelike CMC-$1$ surfaces in ${\Si}^{3}_{1}$ with certain kind of singularities as ``CMC-$1$ faces'' and 
investigated global behavior of CMC-$1$ faces. Moreover, Fujimori, Rossman, Umehara, Yamada and Yang \cite{FRUYY} gave 
the Osserman-type inequality for complete (in the sense of \cite{Fu}) CMC-$1$ faces. As an application of their result, 
we give the upper bound on the number of exceptional values and the totally ramified value number
of the hyperbolic Gauss map for this class (Proposition \ref{face1} and Corollary \ref{face2}). 

The author thanks Professors Ryoichi Kobayashi, Reiko Miyaoka and Junjiro Noguchi for supporting research activities and 
many helpful comments. The author also thanks Professors Shoichi Fujimori, Shin Nayatani, Masaaki Umehara and Kotaro Yamada 
for thier useful advice. 


\section{Preliminaries}
Let $\R^{4}_{1}$ be the Lorentz-Minkowski $4$-space with the Lorentz metric 
\begin{equation}\label{Lmetric}
\langle (x_{0}, x_{1}, x_{2}, x_{3}), (y_{0}, y_{1}, y_{2}, y_{3}) \rangle = -x_{0}y_{0}+x_{1}y_{1}+x_{2}y_{2}+x_{3}y_{3}\;.
\end{equation}
Then the hyperbolic 3-space is 
\[
\H^{3}=\{(x_{0}, x_{1}, x_{2}, x_{3})\in \R^{4}_{1} \, | \, -(x_{0})^{2}+(x_{1})^{2}+(x_{2})^{2}+(x_{3})^{2}=-1, x_{0}>0 \}
\]
with the induced metric from $\R^{4}_{1}$, which is a simply connected Riemannian $3$-manifold with constant 
sectional curvature $-1$. We identify $\R^{4}_{1}$ with the set of $2\times 2$ Hermitian matrices 
Herm($2$)$=\{X^{\ast}=X\}$ $(X^{\ast}:=\tr{\overline{X}}\,)$ by
\begin{equation}\label{Hermite}
(x_{0}, x_{1}, x_{2}, x_{3}) \longleftrightarrow \left(
\begin{array}{cc}
x_{0}+x_{3} & x_{1}+ix_{2} \\
x_{1}-ix_{2} & x_{0}-x_{3}
\end{array}
\right)\, ,
\end{equation}
where $i=\sqrt{-1}$ . In this identification, $\H^{3}$ is represented as 
\begin{equation}\label{hyperbolicspace}
\H^{3}=\{aa^{\ast}\,|\, a\in SL(2,\C)\}
\end{equation}
with the metric 
\[
\langle X, Y \rangle = -\frac{1}{2}\trace{(X\widetilde{Y})}, \quad   \langle X, X \rangle =-\det(X)\, ,
\]
where $\widetilde{Y}$ is the cofactor matrix of $Y$. The complex Lie group $ PSL(2,\C):= SL(2,\C)/\{\pm \text{id} \}$ 
acts isometrically on 
$\H^{3}$ by 
\begin{equation}\label{action}
\H^{3} \ni X \longmapsto aXa^{\ast}\, , 
\end{equation}
where $a\in PSL(2,\C)$. 

Bryant \cite{Br} gave a Weierstrass-type representation formula for CMC-$1$ surfaces. 
\begin{theorem}[Bryant \cite{Br}, Umehara and Yamada \cite{UY1}]\label{W-rep}
Let $\widetilde{M}$ be a simply connected Riemann surface with a reference point $z_{0}\in \widetilde{M}$. 
Let $g$ be a meromorphic function and $\omega$ be a 
holomorphic $1$-form on $\widetilde{M}$ such that 
\begin{equation}\label{metric1}
ds^{2}=(1+|g|^{2})^{2}|\omega|^{2}
\end{equation}
is a Riemannian metric on $\widetilde{M}$. 
Take a holomorphic immersion $F=(F_{ij})\colon \widetilde{M}\to SL(2,\C)$ satisfying $F(z_{0})=\text{id}$ and 
\begin{equation}\label{ODE1}
F^{-1}dF=\left(
\begin{array}{cc}
 g & -g^{2} \\
 1 & -g
\end{array}
\right)\omega\, .
\end{equation}
Then $f\colon \widetilde{M}\to \H^{3}$ defined by 
\begin{equation}\label{imm.}
f=FF^{\ast}
\end{equation}
is a CMC-$1$ surface and the induced metric of $f$ is $ds^{2}$. 
Moreover, the second fundamental form $h$ and the Hopf differential $Q$ 
of $f$ are given as follows: 
\begin{equation}\label{Value}
h=-Q-\overline{Q}+ds^{2}, \quad  Q={\omega}dg\, .
\end{equation}

Conversely, for any CMC-$1$ surface $f\colon \widetilde{M}\to \H^{3}$, there exist a meromorphic function $g$ and 
a holomorphic $1$-form $\omega$ on $\widetilde{M}$ such that the induced metric of $f$ is given by $(\ref{metric1})$ 
and $(\ref{imm.})$ holds, where the map $F\colon \widetilde{M}\to SL(2,\C)$ is a holomorphic null $($``null'' means 
$\det{(F^{-1}dF)=0})$ immersion satisfying $(\ref{ODE1})$.
\end{theorem}

\begin{remark}
Following the terminology of \cite{UY1}, $g$ is called a {\it secondary} Gauss map of $f$. 
The pair $(g, \omega)$ is called {\it Weierstrass data} of $f$, and $F$ is called a {\it holomorphic null lift} of $f$.
\end{remark}

Let $f\colon M\to {\H}^{3}$ be a CMC-$1$ surface of a (not necessarily simply connected) Riemann surface $M$. 
Then the holomorphic null lift $F$ is defined only on the universal cover $\widetilde{M}$ of $M$. Thus,
the Weierstrass data $(g,\omega)$ is not single-valued on $M$. However, the Hopf differential $Q$ of $f$ is 
well-defined on $M$. By (\ref{ODE1}), the secondary Gauss map $g$ satisfies 
\begin{equation}\label{2nd}
g=-\frac{dF_{12}}{dF_{11}}=-\frac{dF_{22}}{dF_{21}}, \quad \text{where}\;  F(z)=\left(
\begin{array}{cc}
F_{11}(z) & F_{12}(z) \\
F_{21}(z) & F_{22}(z)
\end{array}
\right)\, .
\end{equation}
The {\it hyperbolic Gauss map} $G$ of $f$ is defined by, 
\begin{equation}\label{DefofG}
G=\frac{dF_{11}}{dF_{21}}=\frac{dF_{12}}{dF_{22}}\, .
\end{equation}
By identifying the ideal boundary ${\Si}^{2}_{\infty}$ of ${\H}^{3}$ with the Riemann sphere 
$\widehat{\C}:=\C\cup \{\infty\}$, the geometric meaning of $G$ is given as follows (cf.~\cite{Br}): 
The hyperbolic Gauss map $G$ sends each $p\in M$ to the point $G(p)$ at 
${\Si}^{2}_{\infty}$ reached by the oriented normal geodesics of $\H^{3}$ 
that starts at $f(p)$. In particular, $G$ is a meromorphic function on $M$. 

The inverse matrix $F^{-1}$ is also a holomorphic null immersion, and produce a new CMC-$1$ surface 
$f^{\sharp}=F^{-1}(F^{-1})^{\ast}\colon \widetilde{M}\to \H^{3}$, called the {\it dual} of $f$ \cite{UY2}. 
By definition, the Weierstrass data $(g^{\sharp}, {\omega}^{\sharp})$ of $f^{\sharp}$ satisfies 
\begin{equation}
(F^{\sharp})^{-1}dF^{\sharp}=\left(
\begin{array}{cc}
g^{\sharp} & -( g^{\sharp} )^{2} \\
1 & -g^{\sharp}
\end{array}
\right){\omega}^{\sharp}\, . 
\end{equation}
Umehara and Yamada \cite[Proposition~4]{UY2} proved that the Weierstrass data, the Hopf differential $Q^{\sharp}$, and the hyperbolic Gauss map $G^{\sharp}$ of $f^{\sharp}$ 
are given by 
\begin{equation}\label{duality}
g^{\sharp}=G, \quad  {\omega}^{\sharp}=-\frac{Q}{dG},\quad   Q^{\sharp}=-Q,\quad  G^{\sharp}=g \,.
\end{equation}
So this duality between $f$ and $f^{\sharp}$ interchanges the roles of the hyperbolic Gauss map and 
secondary Gauss map.  We call the pair $(G, {\omega}^{\sharp})$ the {\it dual Weierstrass data} of $f$. 
Moreover, these invariants are related by 
\begin{equation}\label{Schwarz}
S(g)-S(G)=2Q\, ,
\end{equation}
where $S(\cdot)$ denotes the Schwarzian derivative
\[
S(h)=\biggl[\biggl(\frac{h''}{h'}\biggr)'-\frac{1}{2}\biggl(\frac{h''}{h'}\biggr)^{2}\,\biggr] dz^{2}, \; \biggl(\, '=\frac{d}{dz}\biggr)
\]
with respect to a complex local coordinate $z$ on $M$. 
By Theorem \ref{W-rep} and (\ref{duality}), the induced metric $ds^{2\sharp}$ of $f^{\sharp}$ is given by 
\begin{equation}\label{dualmetric}
ds^{2\sharp}=(1+|g^{\sharp}|^{2})^{2}|{\omega}^{\sharp}|^{2}=(1+|G|^{2})^{2}\biggl{|}\frac{Q}{dG}\biggr{|}^{2}\, .
\end{equation}
We call the metric $ds^{2\sharp}$ the {\it dual metric} of $f$. There exists the following linkage between 
the dual metric $ds^{2\sharp}$ and the metric $ds^{2}$. 
\begin{lemma}[Umehara-Yamada \cite{UY2}, Z.~Yu \cite{Yu}]\label{complete}
The Riemannian metric $ds^{2\sharp}$ is complete $($resp. nondegenerate$)$ if and only if $ds^{2}$ 
is complete $($resp. nondegenerate$)$.
\end{lemma}

Since $G$ and $Q$ are single-valued on $M$, the dual metric $ds^{2\sharp}$ is also single-valued on $M$. 
So we can define the {\it dual total absolute curvature} 
\[
\text{TA}(f^{\sharp}):=\displaystyle\int_{M} (-K^{\sharp})dA^{\sharp}=\displaystyle\int_{M} \frac{4|dG|^{2}}{(1+|G|^{2})^{2}}\, ,
\]
where $K^{\sharp}$$(\leq 0)$ and $dA^{\sharp}$ are the Gaussian curvature and the area element of $ds^{2\sharp}$ respectively. 
Note that TA($f^{\sharp}$) is the area of $M$ with respect to the (singular) metric induced from the Fubini-Study metric 
on the complex projective line ${\Pi}^{1}(\C)(=\widehat{\C})$ by $G$. 
When the dual total absolute curvature of a complete CMC-$1$ surface is finite, the surface is called 
{\em an algebraic CMC-$1$ surface}. 
\begin{theorem}[Bryant, Huber, Z.~Yu]\label{Huber} An algebraic CMC-$1$ surface $f\colon M \to \H^{3}$ 
satisfies: 
\begin{enumerate}
\item[(i)] $M$ is biholomorphic to $\overline{M}_{\gamma}\setminus \{p_{1},\ldots,p_{k}\}$, where 
$\overline{M}_{\gamma}$ is a closed Riemann surface of genus $\gamma$ and $p_{j}\in\overline{M}_{\gamma}$ $(j=1,\ldots,k)$. $($\cite{Hu}$)$
\item[(ii)] The dual Weierstrass data $(G, {\omega}^{\sharp})$ can be extended meromorphically to $\overline{M}_{\gamma}$\,. $($\cite{Br}$, $\cite{Yu}$)$
\end{enumerate}
\end{theorem}

We call the points $p_{j}$ the {\it ends} of $f$. An end $p_{j}$ of $f$ is called {\it regular} 
if the hyperbolic Gauss map $G$ has at most a pole at $p_{j}$ \cite{UY1}. By Theorem \ref{Huber}, each end of an 
algebraic CMC-$1$ surface is regular. We define the class of ``pseudo-algebraic'' in complete CMC-$1$ surfaces as follows: 

\begin{definition}
A complete CMC-$1$ surface is said to be {\it pseudo-algebraic}, if the following conditions are satisfied: 
\begin{enumerate}
\item[(i)]  The dual Weierstrass data $(G, {\omega}^{\sharp})$ is defined on a punctured Riemann surface 
$M=\overline{M}_{\gamma}\setminus \{p_{1},\ldots,p_{k}\}$, where 
$\overline{M}_{\gamma}$ is a closed Riemann surface of genus $\gamma$ and $p_{j}\in\overline{M}_{\gamma}$ $(j=1,\ldots,k)$. 
\item[(ii)] $(G, {\omega}^{\sharp})$ can be extended meromorphically to $\overline{M}_{\gamma}$. 
\end{enumerate} 
We call $M$ {\it the basic domain} of the pseudo-algebraic CMC-$1$ surface. 
\end{definition} 

Since we do not assume that $f$ is well-defined on $M$, a pseudo-algebraic CMC-$1$ surface is defined on 
the universal cover of $M$. By Theorem \ref{Huber}, algebraic CMC-$1$ surfaces are certainly pseudo-algebraic. 
We give other important examples. 

\begin{example}\label{Vosscousin}
The dual Weierstrass data is defined on $M=\C\setminus \{a_{1},a_{2},a_{3}\}$ for distinct points $a_{1}, a_{2}, a_{3}\in \C$, by 
\begin{equation}\label{Voss}
G=z, \quad {\omega}^{\sharp}=\dfrac{dz}{\prod_{j}(z-a_{j}) }\, .
\end{equation}
As this data does not satisfy the condition that $f$ is well-defined on $M$, 
we get a CMC-$1$ surface $f\colon {\D}\to {\H}^{3}$ on the universal covering disk ${\D}$ of $M$. 
Since the dual metric $ds^{2\sharp}$ is complete, Lemma \ref{complete} implies that  
the metric $ds^{2}$ is also complete. 
Thus we can see that the surface is pseudo-algebraic and the hyperbolic Gauss map omits four values, 
$a_{1}$, $a_{2}$, $a_{3}$ and $\infty$. Starting from $M=\C\setminus \{a_{1},a_{2}\}$, we get similarly a 
pseudo-algebraic CMC-$1$ surface $f\colon {\D}\to {\H}^{3}$, of which hyperbolic Gauss map omits three values, 
$a_{1}$, $a_{2}$ and $\infty$. The completeness of $ds^{2\sharp}$ restricts the number of points $a_{j}'s$ to be less than 
four. We call these surfaces {\it the Voss cousins} because there exist surfaces which have the same Weierstrass data (these surfaces 
are called the Voss surfaces) in minimal surfaces in ${\R}^{3}$ \cite[Theorem 8.3]{O2}. 
\end{example}


\section{Ramification estimate for the hyperbolic Gauss map of pseudo-algebraic CMC-$1$ surfaces} 
We first define the totally ramified value number $\nu_{G}$ of $G$.  
\begin{definition}[Nevanlinna \cite{Ne}]
We call $b\in \widehat{\C}$ a totally ramified value of $G$ when at any 
inverse image of $b$, $G$ branches. 
We regard exceptional values also as totally ramified values. 
Let $\{a_1,\dots,a_{r_o},b_1,\dots,b_{l_0}\}\subset \widehat{\C}$ be the set of 
totally ramified values of $G$, where $a_j$'s are exceptional 
values. For each $a_j$, put $\nu_j=\infty$, and for each $b_j$, 
define $\nu_j$ to be the minimum of the multiplicity of $G$ at 
points $G^{-1}(b_j)$. Then  we have $\nu_j\geq 2$.
We call 
\[
\nu_{G}=\sum_{a_j,b_j}\biggl(1-\dfrac1{\nu_j}\biggr)=r_0+\sum_{j=1}^{l_0}\biggl(1-\dfrac1{\nu_j}\biggr)
\]
{\em the totally ramified value number of $G$}.
\end{definition}

Next, we consider the upper bound on ${\nu}_{G}$. We denote by $D_{G}$ the number of exceptional values of $G$. For the hyperbolic Gauss map of a complete CMC-$1$ surface, 
Z.~Yu \cite{Yu} obtained the upper bound on $D_{G}$. As the application of his argument, we can also get the upper bound on ${\nu}_{G}$ 
for these surfaces. 
\begin{theorem}[Z.~Yu]\label{Zuhan}
Let $f\colon M\to \H^{3}$ be a non-flat complete CMC-$1$ surface and $G$ be the hyperbolic Gauss map of $f$. 
Then we have 
\begin{equation}\label{Yuu}
D_{G}\leq \nu_{G}\leq 4\, .
\end{equation}
\end{theorem}
Note that the Voss cousin (Example \ref{Vosscousin}) shows that the estimate (\ref{Yuu}) is sharp. 
\begin{proof} 
We assume that $M$ is simply connected, otherwise, choose the universal cover of $M$. 
By Lemma \ref{complete}, the dual metric $ds^{2\sharp}$ is also complete. Thus we get a complete minimal surface 
$x\colon M\to {\R}^{3}$ defined by 
\[
x=\Re \int\biggl(1-G^{2},\: i(1+G^{2}),\: 2G\biggr)\,  {\omega}^{\sharp}\, .
\]
In particular, the induced metric of the surface is $ds^{2\sharp}$ and the Gauss map is $G$. By applying 
the Fujimoto theorem \cite[Theorem 1.6.1]{Fuj2}, 
$G$ must be constant if $\nu_{g}>4$. If $G$ is constant, from (\ref{DefofG}), then we obtain 
\begin{equation}\label{Yu1}
F_{11}=GF_{21}+C_{1}, \quad F_{12}=GF_{22}+C_{2}\, ,
\end{equation}
where $C_{1}$ and $C_{2}$ are constant. On the other hand, the secondary Gauss map $g$ of $f$ satisfies 
\begin{equation}\label{Yu2}
g=\frac{GF_{22}-F_{12}}{F_{11}-GF_{21}}\, ,
\end{equation}
because we have
\[
dF=\left(
\begin{array}{cc}
F_{11}g+F_{12} & -g(F_{11}g+F_{12}) \\
F_{21}g+F_{22} & -g(F_{21}g+F_{22})
\end{array}
\right) \omega
\] 
from (\ref{ODE1}). Combining (\ref{Yu1}) and (\ref{Yu2}), we obtain $g=- C_{2}/C_{1}$, that is, $f$ is flat. 
\end{proof}

Moreover, we give more precise estimate for ${\nu}_{G}$ and $D_{G}$ for pseudo-algebraic and algebraic CMC-$1$ surfaces. 
The following is the main result of present paper. 
\begin{theorem}\label{Main1}
Consider a pseudo-algebraic CMC-$1$ surface with the basic domain $M=\overline{M}_{\gamma}\setminus \{p_{1},\ldots,p_{k}\}$. 
Let $d$ be the degree of $G$ considered as a map $\overline{M}_{\gamma}$. Then we have 
\begin{equation}\label{main1}
D_{G}\leq \nu_{G}\leq 2+\frac{2}{R}, \quad \frac{1}{R}=\frac{\gamma -1+k/2}{d}\leq  1\,.
\end{equation}
and for algebraic CMC-$1$ surfaces, we have $R^{-1}<1$. 
\end{theorem}
\begin{remark}
The geometric meaning of ``2'' in the upper bound of (\ref{main1}) is the Euler number of the Riemann sphere. 
The geometric meaning of the ratio $R$ is given in \cite[Section 6]{KKM}.
\end{remark}
\begin{proof}
By definition, ${\omega}^{\sharp}\, (=- Q/dG)$ is single-valued on $M$ and 
can be extended meromorphically to $\overline{M}_{\gamma}$.
Since the dual metric $ds^{2\sharp}$ is nondegenerate, 
the poles of $G$ of order $k$ coincides exactly with the zeros of ${\omega}^{\sharp}$ of order $2k$. By Lemma \ref{complete}, 
the dual metric $ds^{2\sharp}$ is complete, so $\omega^{\sharp}$ has a pole at each $p_{j}$ \cite{O2}. 
By (\ref{duality}), the order of the pole of $\omega^{\sharp}$ 
at $p_{j}$ is given by 
\[
{\mu_{j}}^{\sharp}-d_{j}\geq 1\, , 
\]
where ${\mu_{j}}^{\sharp}\in \Z$ is the branching order of $G$ at $p_{j}$, 
and $d_{j}:=\text{ord}_{p_{j}}Q$. Applying the Riemann-Roch formula to ${\omega}^{\sharp}$ on $\overline{M}_{\gamma}$, 
we obtain 
\[
2d-\displaystyle \sum_{j=1}^{k}({\mu_{j}}^{\sharp}-d_{j})=2{\gamma}-2\,.
\]
Thus we get 
\begin{equation}\label{Ossermanii}
d={\gamma}-1+\frac{1}{2}\sum_{j=1}^{k}({\mu_{j}}^{\sharp}-d_{j})\geq {\gamma}-1+\frac{k}{2}\, ,
\end{equation}
and 
\begin{equation}
R^{-1}\leq 1\, .
\end{equation}
For algebraic case, Umehara and Yamada \cite[Lemma~3]{UY2} showed that the case ${\mu_{j}}^{\sharp}-d_{j}=1$ cannot occur, 
so ${\mu_{j}}^{\sharp}-d_{j}\geq 2$. Thus we get 
\begin{equation}\label{Ossermani}
d={\gamma}-1+\frac{1}{2}\sum_{j=1}^{k}({\mu_{j}}^{\sharp}-d_{j})\geq {\gamma}-1+k > {\gamma}-1+\frac{k}{2}\, ,
\end{equation}
and 
\begin{equation}
R^{-1}<1\, . 
\end{equation}
Now we prove (\ref{main1}). Assume $G$ omits $r_{0}=D_{G}$ values. Let $n_{0}$ be the sum of the branching orders 
at the inverse image of these exceptional values of $G$. We have 
\begin{equation}
k\geq dr_{0}-n_{0}\, .
\end{equation}
Let $b_{1},\ldots,b_{l_{0}}$ be the totally ramified values which are not exceptional values. 
Let $n_{r}$ be the sum of branching order at the inverse image of $b_{i}$ $(i=1,\ldots,l_{0})$ of $G$. 
For each $b_{i}$, we denote 
\[
\nu_i=\text{min}_{G^{-1}(b_i)}\{\text{multiplicity of }G(z)=b_i\}\,,
\]
then the number of points in the inverse image $G^{-1}(b_{i})$ is less than or equal to $d/{\nu}_{i}.$ 
Thus we obtain
\begin{equation}
dl_{0}-n_{r}\leq \displaystyle \sum_{i=1}^{l_{0}}\frac{d}{{\nu}_{i}}\, .
\end{equation}
This implies 
\begin{equation}
l_{0}-\displaystyle \sum_{i=1}^{l_{0}}\frac{1}{{\nu}_{i}}\leq \frac{n_{r}}{d}\, .
\end{equation}
Let $n_{G}$ be the total branching order of $G$ on $\overline{M}_{\gamma}$. 
Then applying the Riemann-Hurwitz theorem to the meromorphic function $G$ on $\overline{M}_{\gamma}$, 
we obtain 
\begin{equation}
n_{G}=2(d+{\gamma}-1)\, .
\end{equation}
Therefore, we get 
\[
{\nu}_{G}\leq r_{0}+\displaystyle \sum_{i=1}^{l_{0}}\biggl(1-\frac{1}{{\nu}_{i}}\biggr)\leq 
\frac{n_{0}+k}{d}+\frac{n_{r}}{d}\leq \frac{n_{G}+k}{d}\leq 2+\frac{2}{R}\, .
\]
\end{proof} 

The system of inequalities (\ref{main1}) is sharp in some cases :
\begin{enumerate}
\item[(i)] The Voss cousins (Example \ref{Vosscousin}) satisfy $d=1$ and ${\gamma}=0$. Then when $k=3$, we have $R^{-1}=1/2$ and
$D_{G}\leq {\nu}_{G}\leq 3$. When $k=4$, we have $R^{-1}=1$ and $D_{G}\leq {\nu}_{G}\leq 4$. 
These show that (\ref{main1}) is sharp in pseudo-algebraic case. 

\item[(ii)] When $(\gamma, k, d)=(0, 1, n)$ $(n\in \N)$, we have 
\[
{\nu}_{G}\leq 2-\frac{1}{n}\, .
\]  
In this case, we can set $M=\C$. We consider the hyperbolic Gauss map $G$ and the Hopf differential $Q$ on $M$, by 
\begin{equation}\label{Ex.1}
G= z^{n}, \qquad Q=\theta z^{n-1}dz^{2} \qquad (\theta\in \C\setminus\{0\})\,.
\end{equation}
Since $M$ is simply connected, we have no period problem. Thus we obtain a solution $g$ of (\ref{Schwarz}) 
defined on $\C$, and then, one can construct an algebraic CMC-$1$ surface 
$f\colon M\to \H^{3}$ with hyperbolic Gauss map and Hopf differential as in (\ref{Ex.1}). (The case $n=1$, 
the surface is congruent to an Enneper cousin dual given by $(g, Q, G)=(\tan{\sqrt{\theta}}z, {\theta}dz^{2}, z)$.) 
In particular, the hyperbolic Gauss map of the surface has $\nu_{G}=2-(1/n)$. Indeed, it has one 
exceptional value, and another totally ramified value of multiplicity $n$ at $z=0$. 
Thus (\ref{main1}) is sharp for algebraic case, too.  
\item[(iii)] When $(\gamma, k, d)=(0, 2, n)$ $(n\in \N)$, we have 
\[
D_{G}\leq {\nu}_{G}\leq 2\, .
\]
In this case, we can set $M=\C\setminus \{0\}$. On the other hand, 
a catenoid cousin ($n=1$) or its $n$-fold cover ($n\geq 2$) 
$f\colon M=\C\setminus \{0\} \to {\H}^{3}$ 
is given by 
\begin{equation}\label{catenoid}
g=\dfrac{n^{2}-l^{2}}{4l} z^{l},\qquad Q=\dfrac{n^{2}-l^{2}}{4z^{2}}dz^{2}, \qquad G=z^{n}\, ,
\end{equation}
where $l\,(\not= n)$ is a positive number. In particular, the hyperbolic Gauss map of these surfaces has $D_{G}={\nu_{G}}=2$. Thus (\ref{main1}) is sharp. 
\end{enumerate}

Now, we consider the case $(\gamma, k, d)=(0, 3, 2). $ By (\ref{main1}), we have ${\nu}_{G}\leq 2.5$.
For minimal surfaces in $\R^{3}$, we can find algebraic minimal surfaces with $\nu_{g}=2.5$ (\cite{Ka1, MS}) and 
the estimate (\ref{main1}) is sharp. However, we cannot find such surfaces for algebraic CMC-$1$ surfaces. 
This property has no analogue in the theory of the Gauss map of algbraic minimal surfaces in ${\R}^{3}$. 
(Note that Rossman and Sato \cite{RS} constructed algebraic CMC-$1$ surfaces (genus one catenoid cousins) 
which have no analogue in the theory of minimal surfaces in $\R^{3}$ \cite{Sc}.)
\begin{proposition}\label{main2}
For the case $(\gamma, k, d)=(0, 3, 2)$, there exist no algebraic CMC-$1$ surfaces with $\nu_{G}=2.5$. 
\end{proposition}
\begin{proof}
The proof is by contradiction. If there exists, 
then the hyperbolic Gauss map $G$ has two exceptional values and another totally ramified value of multiplicity $2$. 
Without loss of generality, we can set $M=\widehat{\C}\setminus \{0, 1, p\}$ $(p\in \C\setminus\{0, 1\})$ 
and $G$ has branch points of branching order $1$ at $z=0$, $\infty$ and $G$ omits two values, $0$ and $1$. 
Then we have $G=z^{2}$ and $p=-1$. By nondegenerateness of $ds^{2\sharp}$ and 
(\ref{Ossermani}), the Hopf differential $Q$ has a pole of order $2$ (resp. order $1$) at $z= \pm 1$ (resp. $z=0$) 
and has no zeros on $\C$. So $Q$ have the form 
\begin{equation}\label{2.5}
Q=\dfrac{{\theta}dz^{2}}{z(z-1)^{2}(z+1)^{2}} \qquad (\theta\in \C\setminus\{0\})\, .
\end{equation}
On the other hand, Rossman, Umehara and Yamada \cite[Theorem~4.5 and 4.6]{RUY1} proved that 
if there exists an algebraic CMC-$1$ surface $f\colon M\to \H^{3}$ whose hyperbolic 
Gauss map and Hopf differential are given by 
\begin{equation}\label{2.5}
G=z^{2}, \qquad Q=\dfrac{{\theta}dz^{2}}{z(z-1)^{2}(z-p)^{2}} \qquad (\theta\in \C\setminus\{0\})\, ,
\end{equation}
then $p$ satisfies ``$p\in\R$ such that 
$p \not= 1$ and $4/(p-1)\not\in \Z$'' or 
``$p=(r+2)/(r-2)$ where $r$ ($\geq 3$)$\in \Z$'', and $\theta=-2p(p+1)$. 
In particular, $p\not= -1$. 
\end{proof}

Finally, we discuss the problem of finding the maximal number of the exceptional values of the hyperbolic Gauss map 
of non-flat algebraic CMC-$1$ surfaces. We call it ``the Osserman problem'' for algebraic CMC-$1$ surfaces. 
As a corollary of Theorem \ref{main1}, we can give some partial results on this problem. 
\begin{corollary}\label{main3}
For non-flat algebraic CMC-$1$ surfaces, we have: 
\begin{enumerate}
\item[(i)] $D_{G}\leq 3$.
\item[(ii)] When ${\gamma}=0$, $D_{G}\leq 2$.  
\item[(iii)] When ${\gamma}= 1$ and the surface has non-embedded regular end, $D_{G}\leq 2$ holds. 
\end{enumerate}
\end{corollary}
\begin{proof}
For algebraic CMC-$1$ surfaces, we have $R^{-1}<1$. Thus we get $D_{G}<4$, that is, $D_{G}\leq 3$.  
Next we prove (ii) and (iii). It is obvious to see that $D_{G}\geq 3$ implies $R^{-1}\geq 1/2$. Thus we get 
\[
{\gamma}-1+\frac{1}{2}\displaystyle \sum_{j=1}^{k}({\mu}_{j}^{\sharp}-d_{j}) \leq 2({\gamma}-1)+k\, . 
\]
As we have ${\mu}_{j}^{\sharp}-d_{j}\geq 2$, it follows 
\begin{equation}\label{D=3}
k\leq \frac{1}{2}\displaystyle \sum_{j=1}^{k}({\mu}_{j}^{\sharp}-d_{j})\leq {\gamma}-1+k\, . 
\end{equation}
Thus we obtain (ii). When ${\gamma}=1$, (\ref{D=3}) implies ${\mu}_{j}^{\sharp}-d_{j}=2$ for all $j$, 
which means all ends are regular and properly embedded (\cite{CHR1,UY2}).
Therefore, if the surface has non-embedded end, then $D_{G}\leq 2$. 
\end{proof}

A catenoid cousin, examples in \cite[Theorem 4.7]{RUY1} and the following proposition 
show that (ii) of Corollary \ref{main3} is sharp. 
\begin{proposition}\label{main4}
There exists an algebraic CMC-$1$ surface $f\colon M=\C\setminus \{0, 1\}\to \H^{3}$ with TA$(f^{\sharp})=-12\pi$ 
whose hyperbolic Gauss map $G$ and Hopf differential $Q$ are given by 
\begin{equation}\label{main3-1}
G=\biggl(\frac{z-1}{z}\biggr)^{3}, \quad Q=\dfrac{{\theta}dz^{2}}{z(z-1)} \quad ({\theta}=-2, -6)\, .
\end{equation} 
In particular, $G$ omits two values, $1$ and $\infty$. 
\end{proposition} 
In general, to find algebraic CMC-$1$ surfaces on a non-simply connected Riemann surface is not so easy. 
Because the period problem, that is, to get a 
proper solution of (\ref{Schwarz}) (see \cite{UY3} or \cite[Lemma 2.1]{RUY1}) on $M$ causes trouble. 
However, Rossman, Umehara and Yamada 
obtained the following construction for the case ${\gamma}=0$ with $n$ ends. 
For the terminology of ordinary differential equation theory, we refer the reader to \cite{CL} and \cite[Appendix A]{RUY1}. 

\begin{lemma}\label{Const}\cite[Proposition 2.2]{RUY1} Let $M=\widehat{\C}\setminus \{p_{1},\ldots,p_{k}\}$ with 
$p_{1},\ldots,p_{k-1}\in \C$. Let $G$ and $Q$ be a meromorphic function and a meromorphic $2$-differential 
on $\widehat{\C}$ satisfying the following two conditions: 
\begin{enumerate}
\item[(a)] For all $q\in M$, $\mathrm{ord}_{q}Q$ is equal to the branching order of $G$, {\rm and}
\item[(b)] for each $p_{j}$, ${\nu}_{j}^{\sharp}-d_{j}\geq 2$\, .
\end{enumerate}
Consider the linear ordinary differential equation 
\begin{equation}\tag{E.0}
\dfrac{d^{2}u}{dz^{2}}+r(z)u=0\, ,
\end{equation}
where $r(z)dz^{2}:=(S(G)/2)+Q$. Suppose $k\geq 2$, and also $d_{j}=\mathrm{ord}_{p_{j}}Q\geq -2$ and 
the indicial equation of $(\mathrm{E.0})$ at $z=p_{j}$ has two distinct roots ${\lambda}_{1}^{(j)}$, ${\lambda}_{2}^{(j)}$ 
and log-term coefficient $c_{j}$, for $j=1, 2, \ldots, k-1$. 
\begin{enumerate}
\item[(i)] Suppose that ${\lambda}_{1}^{(j)}-{\lambda}_{2}^{(j)}\in \Z^{+}$ and $c_{j}=0$ for $j\leq k-1$. 
Then there is exactly a $3$-parameter family of algebraic CMC-$1$ surfaces of $M$ into ${\H}^{3}$ with hyperbolic 
Gauss map $G$ and Hopf differential $Q$. Moreover, such surfaces are $\mathscr{H}^{3}$-reducible. 
\item[(ii)] Suppose that ${\lambda}_{1}^{(j)}-{\lambda}_{2}^{(j)}\in \Z^{+}$ and $c_{j}=0$ for $j\leq k-2$, 
and that ${\lambda}_{1}^{(k-1)}-{\lambda}_{2}^{(k-1)}\in \R\setminus\Z$. Then there exists exactly a $1$-parameter family 
of algebraic CMC-$1$ surfaces of $M$ into ${\H}^{3}$ with hyperbolic Gauss map $G$ and Hopf differential $Q$. 
Moreover, such surfaces are $\mathscr{H}^{1}$-reducible. 
\end{enumerate}
Here we denoted by $\Z^{+}$ the set of positive integers. 
\end{lemma} 
See \cite{RUY0} for the definition of $\mathscr{H}^{1}$-reducible and $\mathscr{H}^{3}$-reducible. Now, we prove Proposition \ref{main4} below.
\begin{proof}[{\rm Proof of Proposition 2.7}]
It is easy to see that $G$ and $Q$ in ($\ref{main3-1}$) satisfy the assumptions (a) and (b) in Lemma \ref{Const}. 
Consider equation (E.0). 
Then the roots of the indicial equations of (E.0) are $-1$ and $2$ at both $z=0$ and at $z=1$. 
By \cite[Appendix A, (A.16)]{RUY1} the log-term coefficients at $z=0$ and at $z=1$ both vanish if and only if 
${\theta}=-2, -6$. By Lemma \ref{Const} (i), the corresponding algebraic CMC-$1$ surfaces are well-defined on $M$. 
\end{proof} 
\begin{remark} Proposition \ref{main4} also provides the following result:  
{\it Any $\mathscr{H}^{3}$-reducible algebraic CMC-$1$ surface that is of type {\bf{O}}$(-1, -1, -2)$ with 
$\mathrm{TA}(f^{\sharp})=-12\pi$ is congruent to an algebraic CMC-$1$ surface $f\colon M=\C\setminus \{0, 1\}\to \H^{3}$ 
whose hyperbolic Gauss map $G$ and Hopf differential $Q$ are as in $($\ref{main3-1}$)$. Here, we say that 
$f$ is a surface of type {\bf{O}}$(d_{1},\ldots, d_{k})$ if $M=\widehat{\C}\setminus \{p_{1},\ldots,p_{k}\}$ and 
$Q$ has order $d_{j}$ at each end $p_{j}$.} 
\end{remark}

However, we do not know whether (i) and (iii) in Corollary \ref{main3} are sharp or not. 
Because we know few algebraic CMC-$1$ surfaces of ${\gamma}\geq 1$ 
and have no definite solution to the global period problem for this class. 


\section{Value distribution of the hyperbolic Gauss map of complete CMC-$1$ faces in de Sitter $3$-space} 
We first briefly recall definitions and basic facts on complete CMC-$1$ faces in de Sitter $3$-space. 
For more details, we refer the reader to \cite{Fu} and \cite{FRUYY}. 
For all of this section, we use the same notation as in Sections $1$ and $2$. 
In the Lorentz-Minkowski $4$-space ${\R}^{4}_{1}$ with the Lorentz metric given by (\ref{Lmetric}), de Sitter $3$-space can be realized by 
\[
{\Si}^{3}_{1}=\{(x_{0}, x_{1}, x_{2}, x_{3})\in \R^{4}_{1} \, | \, -(x_{0})^{2}+(x_{1})^{2}+(x_{2})^{2}+(x_{3})^{2}=1 \}
\]
with metric induced from $\R^{4}_{1}$, which is a simply-connected Lorentzian $3$-manifold with constant sectional curvature 
$1$. By the identification (\ref{Hermite}), ${\Si}^{3}_{1}$ is represented as 
\begin{equation}\label{deSitter1}
{\Si}^{3}_{1}=\{Fe_{3}F^{\ast}\,|\, F\in SL(2,\C)\}, 
\end{equation}
with the metric 
\[
(X, Y) = -\trace{X\widetilde{Y}}-\trace{(Xe_{2}\tr{Y}e_{2})}, 
\]
where $\widetilde{Y}$ is the cofactor matrix of $Y$ and 
\begin{equation}
e_{0}=\left(
\begin{array}{cc}
1 & 0 \\
0 & 1
\end{array}
\right), e_{1}=\left(
\begin{array}{cc}
0 & 1 \\
1 & 0
\end{array}
\right), e_{2}= \left(
\begin{array}{cc}
0 & i \\
-i & 0
\end{array}
\right), e_{3}=\left(
\begin{array}{cc}
1 & 0 \\
0 & -1
\end{array}
\right)\, .
\end{equation}

An immersion into ${\Si}^{3}_{1}$ is said to be {\it spacelike} if the induced metric on the immersed surface is positive definite. 
Aiyama-Akutagawa \cite{AA} gave a Weierstrass-type representation formula for spacelike CMC-$1$ immersions in ${\Si}^{3}_{1}$. 
Moreover, Fujimori \cite{Fu} defined spacelike CMC-$1$ surfaces with certain  kind of singularities as ``CMC-$1$ faces'' and
extended the notion of spacelike CMC-$1$ surfaces.

\begin{definition}
Let $M$ be an oriented $2$-manifold. A $C^{\infty}$-map $f\colon M\to {\Si}^{3}_{1}$ is called a {\it CMC-$1$ face} if 
\begin{enumerate}
\item[($1$)] there exists an open dense subset $W\subset M$ such that $f|_{W}$ is a spacelike CMC-$1$ immersion, 
\item[($2$)] for any singular point (that is, a point where the induced metric degenerates) $p$, there exists a $C^{1}$-differentiable 
function $\lambda\colon U\cap W\to (0,\infty)$, defined on the intersection of neighborhood $U$ of $p$ with $W$, 
such that ${\lambda}ds^{2}$ extends to a $C^{1}$-differentiable Riemannian metric on $U$, where $ds^{2}$ is the 
first fundamental form, i.e., the pull-back of the metric of ${\Si}^{3}_{1}$ by $f$, and 
\item[($3$)] $df(p)\not=0$ for any $p\in M$. 
\end{enumerate}
\end{definition}

A $2$-manifold $M$ on which a CMC-$1$ face $f\colon M\to {\Si}^{3}_{1}$ is defined always has complex 
structure (see \cite{Fu}). So we will regard $M$ as a Riemann surface. The Weierstrass-type representation formula 
in \cite{AA} can be extended for CMC-$1$ faces as follows (\cite[Theorem 1.9]{Fu}): 

\begin{theorem}[Aiyama-Akutagawa, Fujimori]
Let $\widetilde{M}$ be a simply connected Riemann surface with a reference point $z_{0}\in \widetilde{M}$. 
Let $g$ be a meromorphic function and $\omega$ be a 
holomorphic $1$-form on $\widetilde{M}$ such that 
\begin{equation}\label{metric2}
d\hat{s}^{2}=(1+|g|^{2})^{2}|\omega|^{2}
\end{equation}
is a Riemannian metric on $\widetilde{M}$ and $|g|$ is not identically $1$.  
Take a holomorphic immersion $F=(F_{jk})\colon \widetilde{M}\to SL(2,\C)$ satisfying $F(z_{0})=e_{0}$ and 
\begin{equation}\label{ODE2}
F^{-1}dF=\left(
\begin{array}{cc}
 g & -g^{2} \\
 1 & -g
\end{array}
\right)\omega\, .
\end{equation}
Then $f\colon \widetilde{M}\to {\Si}^{3}_{1}$ defined by 
\begin{equation}\label{imm.2}
f=Fe_{3}F^{\ast}
\end{equation}
is a CMC-$1$ face which is conformal away from its singularities. The induced metric $ds^{2}$ on $\widetilde{M}$, 
the second fundamental form $h$ and the Hopf differential $Q$ of $f$ are given as follows: 
\begin{equation}\label{Value2}
ds^{2}=(1-|g|^{2})^{2}|\omega|^{2},\quad h= Q+\overline{Q}+ds^{2}, \quad  Q={\omega}dg\, .
\end{equation}
The singularities of the CMC-$1$ face occur at points where $|g|=1$. 

Conversely, for any CMC-$1$ face $f\colon \widetilde{M}\to {\Si}^{3}_{1}$, there exist a meromorphic function $g$ 
$($with $|g|$ not identically $1$$)$ and a holomorphic $1$-form $\omega$ on $\widetilde{M}$ so that $d{\hat{s}}^{2}$ is a 
Riemannian metric on $\widetilde{M}$ and $(\ref{imm.2})$ holds, where the map $F\colon \widetilde{M}\to SL(2,\C)$ 
is a holomorphic null immersion satisfying $(\ref{ODE2})$. 
\end{theorem}

The holomorphic $2$-differential $Q$ as in (\ref{Value2}) is called the {\it Hopf differential} of $f$. 
In analogy with the theory of CMC-$1$ surfaces in ${\H}^{3}$, the meromorphic function 
\begin{equation}\label{hypG}
G=\dfrac{dF_{11}}{dF_{21}}=\dfrac{dF_{12}}{dF_{22}}
\end{equation}
are called the {\it hyperbolic Gauss map} of $f$. If $f\colon M\to {\Si}^{3}_{1}$ 
be a CMC-$1$ face of a (not necessarily simply connected) Riemann surface $M$, 
then the holomorphic null lift $F$ is defined 
only on the universal cover $\widetilde{M}$ of $M$, but $Q$ and $G$ are well-defined on $M$. 
A geometric meaning of $G$ is given in \cite[Section 4]{FRUYY}. 

Fujimori \cite[Definition 2.1]{Fu} defined the notion of completeness for CMC-$1$ faces as follows: 
We say a CMC-$1$ face $f\colon M\to {\Si}^{3}_{1}$ is {\it complete} if there exists a symmetric $2$-tensor field $T$ 
which vanishes outside a compact subset $C\subset M$ such that the sum $T+ds^{2}$ is a complete Riemannian metric on $M$. 

The completeness characterizes the conformal structure of $M$ \cite[Proposition 1.11]{FRUYY}. 
\begin{theorem}[Fujimori-Rossman-Umehara-Yamada-Yang]\label{conformal}
Let $f\colon M\to{\Si}^{3}_{1}$ be a complete CMC-$1$ face. Then there exist a compact Riemann surface $\overline{M}_{\gamma}$ of genus $\gamma$ 
and a finite number of points $p_{1},\ldots,p_{k}\in\overline{M}_{\gamma}$ such that $M$ is biholomorphic to 
$\overline{M}_{\gamma}\setminus \{p_{1},\ldots,p_{k}\}$. 
\end{theorem}
Any puncture $p_{j}$ is called an {\it end} of $f$. An end $p_{j}$ of $f$ is said to be {\it regular} if $G$ has at most 
a pole at $p_{j}$. Fujimori, Rossman, Umehara, Yamada and Yang 
investigated the behavior of ends and obtained the Osserman-type inequality \cite[Theorem \RII]{FRUYY}\,. 

\begin{theorem}[Fujimori-Rossman-Umehara-Yamada-Yang]\label{FRUYY2}
Suppose a CMC-$1$ face $f\colon M\to{\Si}^{3}_{1}$ is complete. By Theorem \ref{conformal}, there exist a compact Riemann surface $\overline{M}_{\gamma}$ 
of genus $\gamma$ and a finite number of points $p_{1},\ldots,p_{k}\in\overline{M}_{\gamma}$ such that $M$ is biholomorphic to 
$\overline{M}_{\gamma}\setminus \{p_{1},\ldots,p_{k}\}$. Then we have
\begin{equation}\label{FRUYY1}
2d\geq 2{\gamma}-2+2k\, ,
\end{equation}
where $d$ is the degree of $G$ considered as a map $\overline{M}_{\gamma}$ $($if $G$ has essential singularities, then we define 
$d=\infty$ $)$. Furthermore, equality holds if and only if each end is regular and properly embedded.
\end{theorem}

As an application of Theorem \ref{FRUYY2}, we give the upper bound on ${\nu}_{G}$ and $D_{G}$ for complete CMC-$1$ faces. 
When $d$ is finite, that is, each end of $f$ is regular, Theorem \ref{FRUYY2} implies that 
\begin{equation}
\dfrac{1}{R}=\dfrac{{\gamma}-1+k/2}{d}< 1\, .
\end{equation}
Since $G$ can be meromorphically on $\overline{M}_{\gamma}$ for this case, we can apply the latter half of proof of Theorem \ref{Main1}\,. 
Therefore, we have proved the result below.
\begin{proposition}\label{face1}
Let $f\colon M\to{\Si}^{3}_{1}$ be a non-flat complete CMC-$1$ face. 
By Theorem \ref{conformal}, we can assume that $M$ is biholomorphic to $\overline{M}_{\gamma}\setminus \{p_{1},\ldots,p_{k}\}$. 
Let $G$ be the hyperbolic Gauss map of $f$ and $d$ be the degree of $G$ considered as a map $\overline{M}_{\gamma}$. 
Moreover, we assume that each end is regular. Then we have 
\begin{equation}\label{face11}
D_{G}\leq \nu_{G}\leq 2+\frac{2}{R}, \quad \frac{1}{R}=\frac{\gamma -1+k/2}{d}< 1\,.
\end{equation}
In particular, we have $D_{G}\leq 3$. 
\end{proposition}

When $(\gamma, k, d)=(0, 2, 1)$, we have 
\[
D_{G}={\nu}_{G}=2\, .
\]
In this case, there exist examples which show that (\ref{face11}) is sharp. In fact, 
the elliptic catenoid (\cite[Example 5.4]{Fu}) given by
\[
g=z^{\mu},\qquad  Q=\dfrac{1-{\mu}^{2}}{4z^{2}}dz^{2},\qquad  G=z \, ,
\]
where ${\mu}\in \R\setminus \{0\}$ and the parabolic catenoid (\cite[Example 5.6]{FRUYY}) given by
\[
g=\dfrac{\log{z}+1}{\log{z}-1}, \qquad Q=\dfrac{dz^{2}}{4z^{2}}, \qquad G=z \, ,
\]
are complete CMC-$1$ faces defined on ${\C}\setminus \{0\}$ with two regular ends and $D_{G}={\nu}_{G}=2$.

We now obtain the following corollary of Proposition \ref{face1}. 
\begin{corollary}\label{face2}
The hyperbolic Gauss map of a non-flat complete CMC-$1$ face can omit at most three values. 
\end{corollary}
\begin{proof}
If each end is regular, then the hyperbolic Gauss map $G$ can omit at most three values from Proposition \ref{face1}.
If not, then  $G$ has an essential singularity at a puncture. 
By the Big Picard theorem, $G$ attains all but at most two points of $\widehat{\C}$ infinitely often. In particular, 
$G$ can omit at most two values. 
\end{proof}

\end{document}